\documentclass[12pt,reqno]{amsart}

\usepackage{tikz, tikz-cd}
\usetikzlibrary{decorations.pathreplacing,angles,quotes,decorations.pathmorphing,shapes.misc}
\usepackage{bbm,bm}
\usepackage[utf8]{inputenc}
\usepackage[T1]{fontenc}
\usepackage[english]{babel}
\usepackage{amsmath, amsthm}
\usepackage{amssymb} 
\usepackage{comment}
\usepackage{todonotes}
\usepackage{caption}
\usepackage{mathtools}
\usepackage{sidecap}
\usepackage{subcaption}
\usepackage{xcolor}
\usepackage{float}

\usepackage{color}
\newcommand{\blue}[1]{\textcolor{blue}{#1}}

\definecolor{brightpink}{rgb}{1.0, 0.0, 0.5}

\usepackage[bookmarks=true,colorlinks,linkcolor=blue,citecolor=brightpink]{hyperref}
\date{\today}

\theoremstyle{plain}
\newtheorem{theorem}{Theorem}[section]

\newtheorem{lemma}[theorem]{Lemma}
\newtheorem{proposition}[theorem]{Proposition}

\theoremstyle{definition}
\newtheorem{example}[theorem]{Example} 
\newtheorem{remark}[theorem]{Remark} 
\newtheorem{definition}[theorem]{Definition} 
\newtheorem{case}[theorem]{Case}

\specialcomment{moredetails}{\begingroup\color{red}}{\endgroup}

\DeclarePairedDelimiter\Floor\lfloor\rfloor

\DeclareMathOperator{\la}{\langle}
\DeclareMathOperator{\ra}{\rangle}

\DeclareMathOperator{\Z}{\mathbb{Z}}

\DeclareMathOperator{\sube}{\subseteq}

\def\field{\Bbbk}

\author{Nantel Bergeron}\address[Bergeron]
{Department of Mathematics and Statistics\\ York  University\\ To\-ron\-to, Ontario M3J 1P3\\ CANADA}
\email{bergeron@mathstat.yorku.ca}
\urladdr{http://www.math.yorku.ca/bergeron}

\author{Xavier Mootoo}\address[Mootoo]
{Department of Mathematics and Statistics\\ York  University\\ To\-ron\-to, Ontario M3J 1P3\\ CANADA}
\email{xmootoo@gmail.com}

\author{Vedarth Vyas}\address[Vyas]
{Department of Electrical Engineering and Computer Science\\ York  University\\ To\-ron\-to, Ontario M3J 1P3\\ CANADA}
\email{vvyas@my.yorku.ca}

\thanks{Bergeron: Research supported in part by NSERC and York Research Chair.}
\thanks{Mootoo: Research supported by NSERC and Bergeron's York Research Chair}
\thanks{Vyas: Research supported by Bergeron's York Research Chair}

\title[{G}r\"obner Basis of a Catalan Path Ideal]{{T}he Gr\"obner Basis of a Catalan Path Ideal}


\begin{document}

\begin{abstract}
	For the ideal $I = \langle y_1 +  \dots + y_n, y^2_1, \dots , y^2_n \rangle$ in $R = \field[y_1, \dots , y_n]$ 
	 with char($\field$) = 0, we show that the reduced Gr\"obner basis with lex-order consists of polynomials $g_\alpha$ that are
	represented in terms of paths, moving northeast in the Cartesian plane, that stay above the diagonal 
	and cross the diagonal at the last step.  This implies that a linear basis for the quotient ring $R/I$ is given by a set of Catalan paths. We show that the dimension 
	is the number of standard Young tableaux of size $n$ and height at most two. The graded Frobenius characteristic of $R/I$ as a symmetric group module is given by 
	$\sum_{k=0}^{\lfloor \frac{n}{2} \rfloor } s_{n-k,k}q^k$.
\end{abstract} 

\maketitle

\section{Introduction} 
Let $R = \field[y_1, \dots, y_n]$ be the polynomial ring over a field $\field$ with $\text{char}( \field) = 0$ and consider the ideal:
$    I = \la y_1 + y_2 + \dots + y_n, y^2_1, y^2_2, \dots , y^2_n \ra \subseteq R\,.$
We show that the graded quotient $N = R /I$ has its Hilbert series given by 
$$H_N(q)=\sum_{k=0}^{\lfloor \frac{n}{2} \rfloor } f^{(n-k,k)}q^k\,,$$
 where $f^{(n-k,k)}$ is the number of standard Young tableaux of shape $(n-k,k)$, equivalently, the dimension of the irreducible symmetric group indexed by that shape. 
 These numbers are related to paths that remain above the diagonal, such as Catalan paths (see Definition~\ref{catalan_path}).
 We then show that $N$ is indeed a symmetric group module and
its irreducible representation decomposition is as expected (see Section~\ref{sec:irreducible}). This problem was suggested to us by John Machacek~\cite{M}.
This ideal is one among a family of ideals related to those which are studied extensively in \cite{HRS}, and was motivated as a commutative version of the exterior portion of the super-space quotient defined in \cite{Z}. Our main goal is to find the Gr\"obner basis explicitly for the ideal described above.  
\begin{theorem} \label{main_theorem}
	 The reduced Gr\"obner basis with respect to $>_{lex}$ for the ideal $I$ consists of $\{y_2^2, \dots, y_n^2\}$ and of the polynomials
	\begin{equation} \label{main_theorem_eqn}
		g_{\alpha} = y^{\alpha} + \sum_{\beta \in P_{\alpha}} y^{\beta}
	\end{equation}
	Where $\alpha \in \{0,1\}^n$ is a Modified Catalan Path (see Definition~\ref{MCP}) and $P_\alpha$ is a set of paths weakly above $\alpha$ (see Definition~\ref{p_alpha}). \end{theorem}
The   paper is divided as follows. In \S\ref{background}  we introduce some definitions, notations and a loop version of our main theorem; namely Theorem~\ref{main_theorem2}. 
We span \S\ref{proof} to construct the Gr\"obner basis and in \S\ref{sec:irreducible}, we will compute some independent elements in the orthogonal complement of $I$. A dimension argument will allow us to conclude that they form a basis of the quotient and that Theorem~\ref{main_theorem2} holds. As a byproduct, we show that each homogeneous component is of degree $k\le {\lfloor \frac{n}{2} \rfloor }$. It is in fact  an irreducible  symmetric group module indexed by the shape $(n-k,k)$.

\medskip
\noindent
\textbf{Acknowledgement.} We would like to thank John Machacek for his prior work on the project and for his suggestion to continue working on it.

\subsection{ \label{background} Definitions, notation and a loop version of our main theorem}

In the following section, we introduce notation and definitions relevant to the construction of the Gr\"obner basis.
\begin{definition}
	For $\alpha = (\alpha_1, \dots, \alpha_n) \in \Z_{\geq 0}^n$, we denote a monomial $y^\alpha = y_1^{\alpha_1} \cdots y_n^{\alpha_n}\in \field[y_1, \dots, y_n]$. 
	For any square free monomial $y^\alpha$, we call the $\{0,1\}$--sequence $\alpha$, the {\bf bitstring} of the monomial. We often remove comas and parenthesis to represent $\alpha$.
\end{definition}
\begin{definition} \label{catalan_path}
	A {\bf Catalan path} is a series of unit steps forming a staircase walk from $(0,0)$ to $(l,l)$, which is weakly above the line $y=x$ in the Cartesian plane. Each step is either $(0,1)$  an  {\bf east} step, or $(1,0)$ a {\bf north} step. Each path contains a total of $2l$ steps.
\end{definition}
\begin{remark}\label{rem:path}
	   An arbitrary northeast path corresponds to a bit string $\alpha \in \{0, 1 \}^n$, where $\alpha_i = 0$ indicates  the $i^{th}$ step is north, and $\alpha_j = 1$ indicates the $j^{th}$ step 
	   is east. We say that $\alpha$ is a path of length $n$. We have the following correspondence:
	   \begin{center}
		Square-Free Monomial \quad$\iff$\quad Bitstring \quad$\iff$\quad Path.
	   \end{center}
	   We will often reference a monomial, path or bitstring interchangeably as the same object.
\end{remark}
\begin{example}
		Consider the monomials $y_2y_4y_5, y_2y_4y_6 \in \field[y_1, \dots, y_6]$; their bitstring representations are \texttt{010110} and \texttt{010101} respectively. 
		Their corresponding paths are respectively
$$
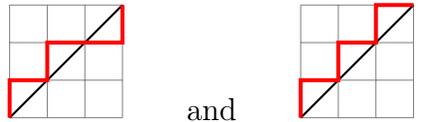
\begin{tikzpicture}[scale=0.5] 
	\draw[help lines] (0,0) grid +(3,3);
	\draw[thick] (0,0) -- +(3,3);
	\draw[line width=1.5pt, draw = red] (0,0)--(0,1)--(1,1)--(1,2)--(2,2)--(3,2)--(3,3);
\end{tikzpicture}
\qquad \text{and} \qquad
\begin{tikzpicture}[scale=0.5] 
	\draw[help lines] (0,0) grid +(3,3);
	\draw[thick] (0,0) -- +(3,3);
	\draw[line width=1.5pt, draw = red] (0,0)--(0,1)--(1,1)--(1,2)--(2,2)--(2,3)--(3,3);
\end{tikzpicture}\,.
$$
\end{example}

For definitions and theorems pertaining to Gr\"obner basis theory, we follow \cite{iva}.  Throughout this paper, we will use the following lexicographic order.
\begin{definition} \label{lex}
	Let $\alpha, \beta \in \Z^n_{\geq 0}$. We say that $\alpha >_{lex} \beta$, if the leftmost nonzero entry of the vector difference $\alpha - \beta \in \Z^n_{\geq 0}$ is positive. 
	We will write $y^\alpha >_{lex} y^\beta$ if $\alpha >_{lex} \beta$. Note that $>_{lex}$ is a monomial order \cite[\S2 \textemdash Proposition 4]{iva}. 
\end{definition}

Theorem~\ref{main_theorem} is our main theorem. To prove it, we decompose the Buchberger's algorithm into loops. The division of S-polynomials produced in each loop is either the next generation of MCP polynomials  or has a zero remainder.
\begin{definition}  \label{MCP}
	A sequence  $\alpha \in \{0,1\}^n$ is a {\bf Modified Catalan Path (MCP)} if there exists integers $l,m\ge 0$ and a Catalan path $w \in \{0,1\}^{2l}$ such that $\alpha = w\texttt{1}(\texttt{0})^m$.
	Note that $\alpha$ is equivalent to a Catalan path with an additional East step after its final step. We let $d=2l+1$ denote the position of the last $1$ in $\alpha$.
\end{definition}
\begin{definition}  \label{pathlength}
	For a sequence $\alpha \in \{0,1\}^n$, we say that $\ell(\alpha)=n$ is the {\bf length} of $\alpha$. The number of  $0$ entries and the number of $1$ entries in $\alpha$ are denoted $	\ell_0(\alpha)$ and $\ell_1(\alpha)$, respectively. Remark that $\ell_1(\alpha)$ is the degree of the corresponding monomial $y^{\alpha}$. Moreover, if $\alpha= w\texttt{1}(\texttt{0})^m$ is an MCP, then 
	$d = \ell(w)+1 =2\ell_1(\alpha)-1$ is the position of the last entry $1$ in $\alpha$.
\end{definition} 
\begin{definition} \label{p_alpha}
	For a fixed MCP $\alpha$, let $d=2\ell_1(\alpha)-1$. We define the set:
	\begin{align*}
		P_{\alpha} = \big\{ \beta \in \{0,1\}^n \big| \beta \neq \alpha,\  \ell_1(\beta)=\ell_1(\alpha) \text{ and } (\alpha_i = 0 \implies \beta_i = 0)_{1 \leq i \leq d}\big\}
	\end{align*}
	Thus, any $\beta \in P_\alpha$ will correspond to a northeast path that is weakly above $\alpha$ and ends in the same position as $\alpha$.
\end{definition}
\begin{remark}\label{rem:lex}
	Note that $\alpha >_{lex} \beta$, for all  $ \beta \in P_{\alpha}$. Hence, the polynomial $g_\alpha$ in Equation~\eqref{main_theorem_eqn} has leading term given by $y^\alpha$.
\end{remark}
\begin{definition} \label{p_alpha_k}
	For an integer $k \geq 1$ such that $\alpha_k=1$, let
		$P_\alpha^{(k)} = \big\{\beta  \in P_\alpha \ | \ \beta_k = 0\big\}.$
\end{definition}
\begin{definition} \label{F_ell}
	Given an integer $\ell \ge -1$ let:
	\begin{equation*}
		F_\ell = \big\{ y^2_2, \dots, y^2_n\big\} \cup \big\{ g_{\alpha} \big|  \alpha \text{ is an MCP and } \ell_1(\alpha)-1 \leq \ell \big\} \,,
	\end{equation*}
	where $g_\alpha$ is as indicated in Equation~\eqref{main_theorem_eqn}. Here $F_{-1}= \big\{ y^2_2, \dots, y^2_n\big\}$.
\end{definition}
\begin{theorem} 
\label{main_theorem2}
    Fix an integer $n>0$ and let $\ell = \Floor{\frac{n-1}{2}}$. The set $F_\ell$ is the reduced Gr\"obner basis for the ideal $I = \la y_1 + y_2 + \cdots + y_n, y_1^2, y_2^2, \dots, y_n^2 \ra$
    \end{theorem}

\section{\label{proof}Constructing $F_\ell$ with S-polynomials}
To prove Theorem~\ref{main_theorem2}, we proceed by iteration, inputting $F_\ell$ into a special loop of Buchberger's algorithm to obtain $F_{\ell +1}$.
Each loop uses only well selected  S-polynomials to produce $F_{\ell +1}$, where Lemma~\ref{lemma_phi_map} is the key to this process.
Once we reach $F_{ \lfloor \frac{n-1}{2} \rfloor}$, we turn to  \S\ref{sec:irreducible} to conclude that all other S-polynomials must have zero remainder  after division by $F_{ \lfloor \frac{n-1}{2} \rfloor}$, concluding our proof. This will be done using a dimension argument.
Remark that $I\ne \la F_\ell \ra$ for $\ell\le 0$, hence, we first prove the following lemma.
\begin{lemma} \label{F_ellsequence}
	Fix an integer $n>0$. Then $I=\la F_1\ra$.
\end{lemma}
\begin{proof} We have that $F_1 = \{g_1, g_{011}, y_2^2, \dots, y_n^2 \}$ where
 	\begin{align*}
 		g_{1} &= y_1 + \sum_{\beta \in P_{1}} y^{\beta} = y_1 + y_2 + \dots + y_n \in I\\
 		g_{011} &= y_2y_3 + \sum_{\beta \in P_{011}}y^\beta = \sum_{2 \leq i < j \leq n}y_iy_j 
	\end{align*}
	To prove {$\la F_1 \ra \sube I $} we need only show that $g_{011} \in I$. Indeed, we have
	$$  g_{011} = \frac{1}{2}(g_1^2 - (y_1^2 + \cdots +y_n^2)) - y_1g_1 + y_1^2  \in I\,.$$
	Conversely, we have   $ y_1^2 =(y_2^2 + \cdots +y_n^2) -g_1^2  + 2(y_1g_1 + g_{011}) \in \la F_1 \ra$.
\end{proof}

Now that this is established, we want to  show that the division algorithm with respect to certain S-polynomials of $F_\ell$ will produce $F_{\ell+1}$.
The next lemma will be very useful for this.
For $g_\alpha\in F_\ell$, we compute {\sl only} the S-polynomials of $g_\alpha$ and $y_j^2\in F_\ell$ for all $j$ such that $\alpha_j=1$, and divide each such S-polynomial by $F_{\ell_1(\alpha)-1}$.
Let us denote the result by
	$$\mathcal{S}_{\alpha, {j}} = \overline{S(g_{\alpha}, y^{2}_{j})}^{F_{\ell_1(\alpha)-1}}$$
\begin{lemma}
	\label{lemma_phi_map} Let $\ell \le \Floor{\frac{n-1}{2}}-1$
	and let $g_\alpha\in F_\ell$.
	Let  $1<k_1 < \dots < k_{\ell_1(\alpha)}= d$ be the positions of all $\alpha_{k_i} = 1$. By definition of $F_\ell$, $\ell_1(\alpha)-1\le\ell$ and $d=2\ell_1(\alpha)-1$.
	There exists an invertible linear map $\color{blue} \phi$ such that:
	\vskip -18pt
	\begin{align*}
	    \textnormal{$\color{blue} \phi$}(\mathcal{S}_{\alpha, k_i}) = \sum_{j=1}^{\ell_1(\alpha)} c_{ij} \mathcal{S}_{\alpha, {k_{j}}} &= g_{\alpha^{(k_i)}}\,,
	\end{align*}
	where $\alpha^{(k_i)} = \alpha_1 \cdot\cdot\, \alpha_{k_{i}-\!1} \textnormal{\texttt{0}} \alpha_{k_{i}+\!1} \cdot\cdot\, \alpha_{d}  
	\textnormal{\texttt{11}} (\textnormal{\texttt{0}})^{n-\!d-\!2}$.
	The $\field$-matrix $M_\alpha=[c_{ij}]$ is invertible.
\end{lemma}

The proof of this lemma is technical and will be done by cases in the sections \S\ref{path_partition}--\S\ref{results}.
Once it is established, the invertibility of the matrix $M_\alpha$ shows that the $g_{\alpha^{(k_i)}}$ can be obtained from the $\mathcal{S}_{\alpha, k_i}$
by a sequence of well chosen elementary Gaussian operations that mimic S-polynomials and divisions. The polynomials $g_{\alpha^{(k_i)}}$ have distinct leading terms 
and are fully reduced; they are in the output of this loop of the Buchberger's algorithm and nothing else is produced from the $\mathcal{S}_{\alpha, k_i}$.
To visualize our case analysis, it will be useful to develop a good graphical representation of the paths.

\medskip
\noindent
\textbf{Visualization.} Given $g_\alpha\in F_\ell$, we have $\alpha = w1(0)^{n-2l-1}$ is an MCP for $\ell(w)= 2l \leq \ell $. 
Let $d=2l+1=2\ell_1(\alpha)-1$ be the position of the last $1$ in $\alpha$.
For $\beta \in P_\alpha^{(k)}$ recall that $\beta_k =0$ and $k$ is chosen such that $\alpha_k=1$. 
Let  $L_{k}$ be the anti-diagonal line intersecting the end of $k^{th}$ position of $\alpha$, and similarly for $L_{d}$  which intersects the last $1$ in $\alpha$. 
Then any such $\beta$ we must have a north step, ending at some position on the line $L_k$. This is visualized as in Example~\ref{p_alpha_k_exmaple}.
\begin{example} \label{p_alpha_k_exmaple}
		Suppose we have $\color{green!60!black} k=4$ and $\color{red}{\alpha} = \texttt{010110$\cdots$0}$. Then $\color{blue!60} d=5$. The path $\color{red}{\alpha}$ and lines  
		$\color{green!60!black} L_k$ and $\color{blue!60} L_{d}$ are visualized as
$$ 
\begin{tikzpicture}[scale=0.4]
	  \draw[help lines] (0,0) grid +(6,6);
  	  \draw[solid, thick] (0,0) -- +(6,6);
	  \draw[draw = red, line width=1.5pt] (0,0)--(0,1)--(1,1)--(1,2)--(2,2)--(3,2);
	  \draw[densely dashed,  line width=1.5pt, draw=red] (3,2) -- (3,6); 
	  \draw[draw=green!60!black,  thick] (0.5, 3.5) -- (4,0);
	  \node at (0.5, 3.5) [circle, fill = green!80!black, scale=0.5]{$\scriptstyle L_k$};
	  \draw[draw=blue!40,  thick] (0.5, 4.5) -- (4,1);
	  \node at (0.5, 4.5) [circle, fill = blue!30, scale=0.5]{$\scriptstyle L_{d}$};
	  \node at (3,2) [circle, fill=blue!40, scale=0.4]{}; 
	  \node at (2,2) [circle, fill=green!60!black, scale=0.4]{}; 
      \end{tikzpicture}
$$
The set $P^{(\text{\textcolor{green!60!black}{4}})}_{\color{red}{\alpha}} $ decomposes into four different possibilities depending on the first $d$ entries of  $\beta=\beta_1\beta_2\beta_3\beta_4\beta_5\cdots$:
$$ P^{(\text{\textcolor{green!60!black}{4}})}_{\color{red}{\alpha}} = \big\{
	\texttt{010\textcolor{green!60!black}{0}1} |  \cdot \cdot  \texttt{1}  \cdots;\ 
	\texttt{010\textcolor{green!60!black}{0}0} |  \cdot \cdot  \texttt{1} \cdot \cdot  \texttt{1} \cdots ;\ 
	\texttt{000\textcolor{green!60!black}{0}1} |  \cdot \cdot  \texttt{1} \cdot \cdot \texttt{1}  \cdots;\ 
	\texttt{000\textcolor{green!60!black}{0}0} |  \cdot \cdot  \texttt{1}  \cdot \cdot  \texttt{1} \cdot \cdot \texttt{1} \cdots\big\}
$$ 
We visualize this as	
$$
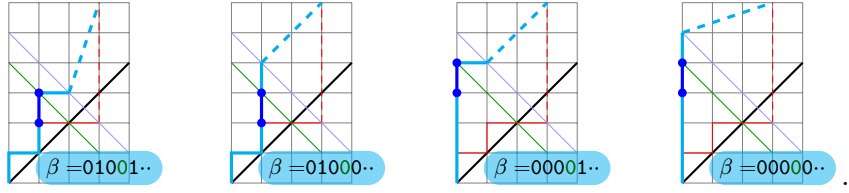
\begin{tikzpicture}[scale=0.4]
	 \draw[help lines] (0,0) grid +(4,6);
  	 \draw[solid, thick] (0,0) -- +(4,4);
	 \draw[draw = red] (0,0)--(0,1)--(1,1)--(1,2)--(2,2)--(3,2);
	 \draw[densely dashed, draw=red] (3,2) -- (3,6); 
	 \draw[draw=green!60!black] (0,4) -- (3,1);
	 \draw[draw=blue!40] (0,5) -- (4,1);
	\draw[draw = cyan,very thick] (0,0)--(0,1)--(1,1)--(1,2);
	\draw[draw = blue!100, very  thick] (1,2) -- (1,3); 
	\draw[draw = cyan,very thick] (1,3)--(2,3);
	\draw[draw = cyan, dashed, very thick] (2,3) -- (3,6); 
	\node at (3, 0.5) [rounded rectangle, opacity = 0.5, fill = cyan, scale=0.7]{$\beta = $\texttt{010\textcolor{green!40!black}{0}1}$\cdot\cdot$};
	\node at (3, 0.5) [scale=0.7]{$\beta = $\texttt{010\textcolor{green!40!black}{0}1}$\cdot\cdot$};
	\node at (1, 2) [circle, fill=blue, scale=0.3]{}; 
	\node at (1, 3) [circle, fill=blue, scale=0.3]{}; 
\end{tikzpicture} \,
\qquad
\begin{tikzpicture}[scale=0.4]
	\draw[help lines] (0,0) grid +(4,6);
	\draw[solid, thick] (0,0) -- +(4,4);
	\draw[draw = red] (0,0)--(0,1)--(1,1)--(1,2)--(2,2)--(3,2);
	\draw[densely dashed, draw=red] (3,2) -- (3,6); 
	\draw[draw=green!60!black] (0,4) -- (3,1);
	\draw[draw=blue!40] (0,5) -- (4,1);
	\draw[draw = cyan,very thick] (0,0)--(0,1)--(1,1)--(1,2);
	\draw[draw = blue!100, very  thick] (1,2) -- (1,3); 
	\draw[draw = cyan,very thick] (1,3)--(1,4);
	\draw[draw = cyan, dashed, very thick] (1,4) -- (3,6); 
	\node at (3, 0.5) [rounded rectangle, opacity = 0.5, fill = cyan, scale=0.7]{$\beta = $\texttt{010\textcolor{green!40!black}{0}0}$\cdot\cdot$};
	\node at (3, 0.5) [scale=0.7]{$\beta = $\texttt{010\textcolor{green!40!black}{0}0}$\cdot\cdot$};
	\node at (1, 2) [circle, fill=blue, scale=0.3]{}; 
	\node at (1, 3) [circle, fill=blue, scale=0.3]{}; 
\end{tikzpicture}\,
\qquad
\begin{tikzpicture}[scale=0.4]
	\draw[help lines] (0,0) grid +(4,6);
	\draw[solid, thick] (0,0) -- +(4,4);
	\draw[draw = red] (0,0)--(0,1)--(1,1)--(1,2)--(2,2)--(3,2);
	\draw[densely dashed, draw=red] (3,2) -- (3,6); 
	\draw[draw=green!60!black] (0,4) -- (3,1);
	\draw[draw=blue!40] (0,5) -- (4,1);
	\draw[draw = cyan,very thick] (0,0)--(0,1)--(0,2)--(0,3);
	\draw[draw = blue!100, very  thick] (0,3) -- (0,4); 
	\draw[draw = cyan,very thick] (0,4)--(1,4);
	\draw[draw = cyan, dashed, very thick] (1,4) -- (3,6); 
	\node at (3, 0.5) [rounded rectangle, opacity = 0.5, fill = cyan, scale=0.7]{$\beta = $\texttt{000\textcolor{green!40!black}{0}1}$\cdot\cdot$};
	\node at (3, 0.5) [scale=0.7]{$\beta = $\texttt{000\textcolor{green!40!black}{0}1}$\cdot\cdot$};
	\node at (0, 3) [circle, fill=blue, scale=0.3]{}; 
	\node at (0, 4) [circle, fill=blue, scale=0.3]{}; 
\end{tikzpicture}\,
\qquad
\begin{tikzpicture}[scale=0.4]
	\draw[help lines] (0,0) grid +(4,6);
	\draw[solid, thick] (0,0) -- +(4,4);
	\draw[draw = red] (0,0)--(0,1)--(1,1)--(1,2)--(2,2)--(3,2);
	\draw[densely dashed, draw=red] (3,2) -- (3,6); 
	\draw[draw=green!60!black] (0,4) -- (3,1);
	\draw[draw=blue!40] (0,5) -- (4,1);
	\draw[draw = cyan,very thick] (0,0)--(0,1)--(0,2)--(0,3);
	\draw[draw = blue!100, very  thick] (0,3) -- (0,4); 
	\draw[draw = cyan,very thick] (0,4)--(0,5);
	\draw[draw = cyan, dashed, very thick] (0,5) -- (3,6); 
	\node at (3, 0.5) [rounded rectangle, opacity = 0.5, fill = cyan, scale=0.7]{$\beta = $\texttt{000\textcolor{green!40!black}{0}0}$\cdot\cdot$};
	\node at (3, 0.5) [scale=0.7]{$\beta = $\texttt{000\textcolor{green!40!black}{0}0}$\cdot\cdot$};
	\node at (0, 3) [circle, fill=blue, scale=0.3]{}; 
	\node at (0, 4) [circle, fill=blue, scale=0.3]{}; 
\end{tikzpicture} \,.  
$$
\end{example}
	
From Example~\ref{p_alpha_k_exmaple}, it is evident that $P^{(4)}_{\alpha}$ decomposes  according to the first $d$ entries of $\beta\in P^{(4)}_{\alpha}$. The remaining entries of  
$\beta$ are only restricted by the number of ones. Also note that the number of zeros in $\beta_1 \cdots \beta_{d}$  determines the position where $\beta$ crosses the line $L_{d}$. Indeed, let
$$	\ell_0(\beta_1 \cdots \beta_{d}) = q\quad\text{and}\quad 	\ell_1(\beta_1 \cdots \beta_{d} ) = d - q \,,$$
then the path $\beta$ intersects the line $L_{d}$ at the point $(q, d-q)$.

\subsection{Decomposing $P_\alpha$} \label{path_partition}
In this section, we decompose the set $P_\alpha$ according to first $d=2\ell_1(\alpha)-1$ entries of the path $\beta\in P_\alpha$.
\begin{definition} \label{u_alphar}
	Let $\alpha$ be an MCP with $d = 2\ell_1(\alpha)-1$. For $1 \leq r \leq \ell_1(\alpha)$, we define
		\begin{equation}
		U_{\alpha, r} = \big \{ u\in \{0,1\}^d \big | ( \alpha_i = 0 \implies u_i = 0)_{1 \leq i \leq d}\  \text{ and }\   \ell_1(u) = \ell_1(\alpha) -r  \big \}.
		 \label{u_alpha_eqn}
	\end{equation}
	Given $u \in U_{\alpha,r}$, we define
	\begin{equation}
		P_{u, \alpha} = \big\{ \beta\in P_\alpha \big|  \beta_{1} \cdots \beta_d=u  \big\}. 
		\label{p_u_alpha_eqn}
	\end{equation}
\end{definition}
\begin{example} \label{u_alpha_example}
	Let $\alpha = \texttt{010110} \cdots \texttt{0}$, then $\ell_1(\alpha)=3$. For $r =1$, we have
		\begin{equation*} 
		U_{\texttt{010110} \cdots \texttt{0},1} = \{\texttt{01010, 01001, 00011}\}\,.
	\end{equation*}
\end{example} 
The subsets $P_{u, \alpha}$ are disjoint, so we can partition $P_\alpha = \bigcup_{r = 1}^{\ell_1(\alpha)} \bigcup_{u \in U_{\alpha,r}} P_{u,\alpha}$. Therefore, 
\begin{equation} \label{y_beta_sum}
	\sum_{\beta \in P_\alpha} y^ \beta   = \sum_{r=1}^{\ell_1(\alpha)} \sum_{u \in U_{\alpha,r}} y^u
										\sum_{d+1 \leq i_1 < \cdots < i_r \leq n} y_{i_1} \cdots y_{i_{r}}   \,.
\end{equation} 
We remark that for $\beta\in P_{u, \alpha}$, the only restriction on $\beta_{d+1}\cdots\beta_n$ is that $\ell_1(\beta_{d+1}\cdots\beta_n)=\ell_1(\beta)-\ell_1(u)=r$.
Hence, in the last  summation in Equation~\eqref{y_beta_sum},  we are summing over all square-free monomials of degree $r$ in the variables $y_{d+1},\dots,y_n$.

%
%
%

\begin{definition} \label{u_alpha_one_k}
     For $\alpha_k=1$, let $u(\alpha, k) = \alpha_1 \cdots \alpha_{k-1}\texttt{0}\alpha_{k+1} \cdots \alpha_{d}\in U_{\alpha, 1}$.
\end{definition}

\subsection{Computing  $\overline{S(g_\alpha, y_k^2)}^{F_{-1}} $} \label{special_paths} We first investigate the S-polynomial we want to compute, dividing only by $F_{-1}=\{y_2^2,\ldots,y_n^2\}$.
\begin{lemma}
	    \label{spoly_lemma}
	    Let $g_\alpha \in F_\ell$ and
	     given any $1< k\le 2\ell_1(\alpha)-1$, such that $\alpha_k=1$.  ($\alpha_1$ is always 0). We have
	    \begin{equation} \label{spoly_lemma_eqn}
    	    	\overline{S(g_\alpha, y_k^2)}^{F_{-1}} = \sum_{\beta \in P_{\alpha}^{(k)}} y^\beta y_k.
	    \end{equation}
    \end{lemma} 
\begin{proof} 
	Since $\alpha_k =1$, we have $\text{LCM}(\text{LM}(g_\alpha), y_k^2) = y^\alpha y_k$ and the result is as follows
	$$
		\overline{S(g_\alpha, y_k^2)}^{F_{-1}}   =  \overline{y_k g_\alpha -y^\alpha y_k}^{F_{-1}}  
		=  \overline{\sum_{\beta \in P_{\alpha}} y^\beta y_k }^{F_{-1}} 
		=\sum_{\beta \in P_{\alpha}^{(k)}} y^\beta y_k\,.
	$$
	The last equality follows from the fact that any monomial $\overline{y^\beta y_k }^{F_{-1}}=0$, if and only if $\beta\in P_\alpha\setminus  P_{\alpha}^{(k)}$.
\end{proof}

For $g_\alpha\in F_\ell$ and $k$ such that
$\alpha_k = 1$, Lemma~\ref{spoly_lemma} and $P_\alpha^{(k)} = \bigcup_{r = 1}^{\ell_1(\alpha)} \bigcup_{u \in U_{\alpha,r},\,u_k=0 } P_{u,\alpha}$ gives
\begin{equation}
 \label{spoly_lemma_eqn2}
	\overline{S(g_\alpha, y_k^2)}^{F_{-1}} = \sum_{\beta \in P_\alpha^{(k)}} y^\beta y_k  = \sum_{r=1}^{\ell_1(\alpha)}\sum_{u \in U_{\alpha,r}\atop u_k=0 \hfill} y^{u}y_k \sum_{d+1 \leq i_1 < \cdots < i_r \leq n} y_{i_1} \cdots y_{i_{r}}  
\end{equation} 
whereas before $d=2\ell_1(\alpha)-1$. Now we need to divide $\overline{S(g_\alpha, y_k^2)}^{F_{-1}}$ by $F_{\ell_1(\alpha)-1}$. We will do this in two separate cases: the terms of this equation for $r=1$ and those for $r>1$.

\subsection{Division of the $r=1$ term in Equation~\eqref{spoly_lemma_eqn2} by $F_{\ell_1(\alpha)-1}$}  \label{beta_prime_redu} 
When $r=1$ in Equation~\eqref{spoly_lemma_eqn2}, there is only a single $u\in U_{\alpha,1}$ such that $u_k=0$, namely $u(\alpha,k)$ as in Definition~\ref{u_alpha_one_k}.
The $r=1$ term in Equation~\eqref{spoly_lemma_eqn2} is 
$$
	 y^{u(\alpha,k)}y_k \sum_{d+1 \leq j \leq n}  y_{j}  = \sum_{j=d+1}^n y^\alpha y_j 
$$
The terms in the sum above are monomial factors of $\text{LT}(g_\alpha) = y^\alpha$.  Using the Division Algorithm we can divide the expression further using $ g_\alpha$:
    \begin{align*}
    \overline{ \sum_{j=d+1}^n y^\alpha y_j }^{\{g_\alpha\}} &=  \sum_{j=d+1}^n y^\alpha y_j - g_\alpha \sum_{j=d+1}^n y_j  \\
    &=  -\sum_{\gamma \in P_\alpha}y^\gamma \sum_{j=d+1}^n y_j \\ 
    &= - \sum_{s=1}^{\ell_1(\alpha)} \sum_{v \in U_{\alpha,s}} y^v \sum_{d+1 \leq i_1 < \cdots < i_s \leq n}y_{i_1} \cdots y_{i_s} \sum_{j=d+1}^n y_j
\end{align*}
We use Equation~\eqref{y_beta_sum} in the last equality. We can further divide this last expression using $F_{-1}$ and we get
\begin{equation}\label{eq:thefirstr1}
    \overline{ \sum_{j=d+1}^n y^\alpha y_j }^{\{g_\alpha,F_{-1}\}} 
    = - \sum_{s=1}^{\ell_1(\alpha)} (s+1) \sum_{v \in U_{\alpha,s}} y^v \sum_{d+1 \leq i_1 < \cdots < i_{s+1} \leq n}y_{i_1} \cdots y_{i_s} y_{i_{s+1}} \\
\end{equation}
This result follows from the fact that 
\begin{align*}
   \sum_{d+1 \leq i_1 < \cdots < i_s \leq n} \hskip-12pt y_{i_1} \cdots y_{i_s} \sum_{j=d+1}^n y_j = (s+1)\hskip-24pt \sum_{d+1 \leq i_1 < \cdots < i_{s+1} \leq n} \hskip-12pt y_{i_1} \cdots y_{i_s} y_{i_{s+1}} + \text{terms containing $y_j^2$}
\end{align*}
Once again, we want to separate the $r=1$ terms in Equation~\eqref{eq:thefirstr1} and those for $r \geq 2$.

\begin{case}[{\it $s=1$ in Equation~\eqref{eq:thefirstr1}}\null] \label{r=1}
In the case of $s = 1$ in Equation~\eqref{eq:thefirstr1} and $v \in U_{\alpha, 1}$, we can represent any path $ y^{v}y_{i_1}y_{i_2}$ visually as in
Figure~\ref{Fig:case1}.
For each $v \in U_{\alpha, 1}$, we have $\ell_1(v) =\ell_1(\alpha) - 1$ and let $k$ be the unique position, where $v_k=0$ and $\alpha_k=1$.
The paths of the monomial $y^{v}y_{i_1}y_{i_2}$ in the sum have two ones after the $d^{th}$ position.
If we select $i_1 = d+1$ and $i_2 = d+2$, this will form a  {\sl new} MCP path
$\alpha^{(k)}=\alpha_1 \cdot\cdot\, \alpha_{k-\!1} \textnormal{\texttt{0}} \alpha_{k+\!1} \cdot\cdot\, \alpha_{d}  \textnormal{\texttt{11}} (\textnormal{\texttt{0}})^{n-\!d-\!2}$, 
as displayed above. All the other paths in this case remain above the diagonal. There are no more possible divisions using $F_{\ell_1(\alpha)-1}$.
The MCP path is actually new for $F_\ell$ if $\ell_1(\alpha)-1=\ell$ as  $\ell_1(\alpha^{(k)})-1>\ell$.
\begin{figure}[htpb]
\begin{tikzpicture}[scale = 0.4, trim left=-2cm]
	\draw[help lines] (0,0) grid +(9,9);
	\draw[solid, thick] (0,0) -- +(9,9);
	\draw[draw = red,  thick] (0,0) -- (0,2) -- (1,2)--(1,3) -- (2,3)--(2,4)--(5,4); 
	\draw[draw = red,  thick, densely dashed] (5,4) -- (5,9); 
  	\draw[draw = cyan, opacity=0.4, line width = 3pt] (0,0) -- (0,2); 
  	\draw[draw=cyan, very thick, dashed, -latex] plot [smooth] coordinates {(0,2) (1.3, 4)  (4,5)}; 
	\draw[draw=blue!40, very thick] (0, 9) -- (5,4) -- (7.5, 1.5); 
 	\node at (7.5, 1.5) [circle, fill = blue!30, scale=0.5]{$ L_{d}$}; 
  	\draw [->,decorate, thick, decoration={snake,amplitude=.4mm,segment length=2mm,post length=1mm}] (0,6) -- (1.3,4);
 	\node at (0,6) [circle, fill = cyan!40, draw = black, scale = 1]{{$\scriptstyle v$}};
	\draw [->,decorate, thick, decoration={snake,amplitude=.4mm,segment length=2mm,post length=1mm}] (3,7.5) -- (4.5,6);
	\node at (3,7.5) [rounded rectangle, , fill = cyan!40, draw = black, scale = 0.7]{{$y_{i_1}y_{i_2}$}};
	\draw[draw = violet, very thick] (4,5) -- (6,5); 
	\draw[draw = cyan, dashed, very thick, -latex] (4,5) -- (6,8); 
	\draw[draw = cyan, dashed, very thick] (6,5) -- (6,8); 
	\node at (4, 5) [circle, fill = cyan, scale=0.4]{};
	\node at (6, 5) [circle, fill = violet, scale=0.4]{};
	\draw [->,decorate, thick, decoration={snake,amplitude=.4mm,segment length=2mm,post length=1mm}] (10,5) -- (6,5);
	\node at (11, 5) [rounded rectangle, draw = violet, fill = white, scale = 0.9]{$s=1\colon \text{ New MCP for}\atop  y_{i_1}y_{i_2}=y_{d+1}y_{d+2}$};
\end{tikzpicture}
\ 
\begin{tikzpicture}[scale = 0.4, trim left=-2cm]
	\draw[help lines] (0,0) grid +(9,9);
	\draw[solid, thick] (0,0) -- +(9,9);
	\draw[draw = red,  thick] (0,0) -- (0,2) -- (1,2)--(1,3) -- (2,3)--(2,4)--(5,4); 
	\draw[draw = red,  thick, densely dashed] (5,4) -- (5,9); 
  	\draw[draw = cyan, opacity=0.4, line width = 3pt] (0,0) -- (0,2); 
	\draw[draw=blue!40, very thick] (0,9) -- (5,4) -- (7.5, 1.5); 
 	\node at (7.5, 1.5) [circle, fill = blue!30, scale=0.5]{$ L_{d}$}; 
	\draw[draw=blue!50!green,  thick, dashed, -latex] plot [smooth] coordinates {(0,2) (1.3, 5)  (3,6)}; 
	\draw[draw=orange,  thick, dashed, -latex] plot [smooth] coordinates {(0,2) (0.8, 5.6)  (2,7)}; 
	\draw[draw=magenta,  thick, dashed, -latex] plot [smooth] coordinates {(0,2) (0.4, 5.8)  (1,8)}; 
	\draw[draw=blue, thick, dashed, -latex] plot [smooth] coordinates {(0,2) (0.2, 5.9)  (0.4,7.8)}; 
	\draw[draw=blue!50!green,  thick, dashed, -latex] (3,6)--(6,8); 
	\draw[draw=orange,  thick, dashed, -latex] (2,7)--(6,8); 
	\draw[draw = blue!50!green,  thick] (3,6) -- (6,6); 
	\draw[draw = orange,  thick] (2,7) -- (6,7); 
	\node at (3, 6) [circle, fill = blue!50!green, scale=0.4]{}; 
	\node at (6, 6) [circle, fill = blue!50!green, scale=0.4]{}; 
	\node at (2, 7) [circle, fill = orange, scale=0.4]{}; 
	\node at (6, 7) [circle, fill = orange, scale=0.4]{}; 
	\node at (10, 5.9)[rounded rectangle, draw = blue!50!green, fill = white, scale = 0.7]{{$s=2\colon y_{i_1}y_{i_2}y_{i_3}$}};
	\node at (10.4, 7.1)[rounded rectangle, draw = orange, fill = white, scale = 0.7]{{$s=3\colon y_{i_1}y_{i_2}y_{i_3}y_{i_4}$}};
	\node at (10, 8.5)[scale = 0.6]{$\mathbf \vdots$};
	\draw[draw = cyan, dashed, very thick] (6,6) -- (6,8); 
\end{tikzpicture}
            \caption{On the left is Case~\ref{r=1} and on the right is Case~\ref{rgeq2}}
            \label{Fig:case1}
\end{figure} 
\end{case}

\begin{case}[{\it $s \geq 2$ in Equation~\eqref{eq:thefirstr1}}\null] \label{rgeq2}
    Similarly, in the case of $s \geq 2$ and $v \in U_{\alpha, s}$, we can visualize any such path $y^{v}y_{i_1} \cdots y_{i_{s+1}}$ as in Figure~\ref{Fig:case1}.
All such paths clearly remain above the diagonal and no further divisions using $F_{\ell_1(\alpha)-1}$ are possible.
\end{case}	

\subsection{Division of $r>1$ terms in Equation~\eqref{spoly_lemma_eqn2} by $F_{\ell_1(\alpha)-1}$} \label{beta_redu}
We now focus on the  $r>1$ terms in Equation~\eqref{spoly_lemma_eqn2}. 
Fix $r>1$ and fix $u \in U_{\alpha, r}$ such that $u_k=0$. We are considering a monomial of the form 
$y^{u}y_k y_{i_1} \cdots y_{i_{r}}  $, where $d+1 \leq i_1 < \cdots < i_r \leq n$.
Note that $\alpha_k = 1$ and $u_k =0$, and there is at least one more entry such that $\alpha_j = 1$ and $u_j = 0$, for $j \ne k$. Hence, there are \textit{at least} two entries in which $u$ differs from $\alpha$. We further split $r>1$ terms into two cases, where $r=2$ or $r>2$.
We visualize this in Figure~\ref{Fig:case2}.
\begin{figure}[htpb]
\begin{align*}
\begin{tikzpicture}[scale = 0.4,trim left=-2cm]
	\draw[help lines] (0,0) grid +(9,9);
	\draw[solid, thick] (0,0) -- +(9,9);
	\draw[draw = red,  thick] (0,0) -- (0,2) -- (1,2)--(1,3) -- (2,3)--(2,4)--(5,4); 
	\draw[draw = red,  thick, densely dashed] (5,4) -- (5,9); 
  	\draw[draw = cyan, opacity=0.4, line width = 3pt] (0,0) -- (0,2); 
	\draw[draw=blue!40,  thick] (0,9) -- (5,4) -- (7.5, 1.5); 
 	\node at (7.5, 1.5) [circle, fill = blue!30, scale=0.5]{$ L_{d}$}; 
	\draw[draw=green!60!black,  thick] (0,7) -- (5.5,1.5); 
 	\node at (5.5,1.5) [circle, fill = green!80!black, scale=0.5]{$ L_{k}$}; 
	\draw[draw=cyan,  thick, dashed, -latex] plot [smooth] coordinates {(2,5) (3,6)}; 
	\draw[draw=cyan,  thick, dashed, -latex] plot [smooth] coordinates {(1,6) (2, 6)  (3,6)}; 
	\draw[draw=cyan,  thick, dashed, -latex] plot [smooth] coordinates {(0,2) (1,3.5)  (2,4)}; 
	\draw[draw=cyan,  thick, dashed, -latex] plot [smooth] coordinates {(0,2) (0.5, 4) (1,5)}; 
	\draw[draw=cyan,  thick, dashed, -latex] (3,6)--(5,9); 
	\draw[draw = cyan,  thick] (3,6) -- (5,6); 
	\draw[draw = cyan,  thick] (2,4) -- (2,5); 
	\draw[draw = cyan,  thick] (1,5) -- (1,6); 
	\node at (3, 6) [circle, fill = cyan, scale=0.4]{}; 
	\node at (5, 6) [circle, fill = cyan, scale=0.4]{}; 
	\node at (2,5) [circle, fill = cyan, scale=0.4]{}; 
	\node at (1,6) [circle, fill = cyan, scale=0.4]{}; 
	\node at (4.5,-1)[rounded rectangle, draw = blue!50!green, fill = white, scale = 0.7]{{$r=2\colon y^u y_{i_1}y_{i_2}$}};
\end{tikzpicture}
&\begin{tikzpicture}[scale = 0.4, trim left=-2cm]
	\draw[help lines] (0,0) grid +(9,9);
	\draw[solid, thick] (0,0) -- +(9,9);
	\draw[draw = red,  thick] (0,0) -- (0,2) -- (1,2)--(1,3) -- (2,3)--(2,4)--(5,4); 
	\draw[draw = red,  thick, densely dashed] (5,4) -- (5,9); 
  	\draw[draw = cyan, opacity=0.4, line width = 3pt] (0,0) -- (0,2); 
	\draw[draw=blue!40,  thick] (0,9) -- (5,4) -- (7.5, 1.5); 
 	\node at (7.5, 1.5) [circle, fill = blue!30, scale=0.5]{$ L_{d}$}; 
	\draw[draw=green!60!black,  thick] (0,7) -- (5.5,1.5); 
 	\node at (5.5,1.5) [circle, fill = green!80!black, scale=0.5]{$ L_{k}$}; 
	\draw[draw=cyan,  thick, dashed, -latex] plot [smooth] coordinates {(3,4)   (4,5)}; 
	\draw[draw=cyan,  thick, dashed, -latex] plot [smooth] coordinates {(2,5) (4,5)}; 
	\draw[draw=cyan,  thick, dashed, -latex] plot [smooth] coordinates {(0,2) (1,3.5)  (2,4)}; 
	\draw[draw=cyan,  thick, dashed, -latex] plot [smooth] coordinates {(0,2) (0.5, 4) (1,5)}; 
	\draw[draw=cyan,  thick, dashed, -latex] (4,5)--(6,8); 
	\draw[draw=cyan,  thick, dashed] (6,5)--(6,8); 
	\draw[draw = cyan!70!black, ultra thick] (4,5) -- (6,5); 
	\draw[draw = cyan,  thick] (2,4) -- (3,4); 
	\draw[draw = cyan,  thick] (1,5) -- (2,5); 
	\node at (4, 5) [circle, fill = cyan, scale=0.4]{}; 
	\node at (6, 5) [circle, fill = cyan, scale=0.4]{}; 
	\node at (2,5) [circle, fill = cyan, scale=0.4]{}; 
	\node at (3,4) [circle, fill = cyan, scale=0.4]{}; 
	\node at (4.5,-1)[rounded rectangle, draw = blue!50!green, fill = white, scale = 0.7]{{$\displaystyle r=2\colon y^u y_k y_{i_1}y_{i_2} 
								\atop \displaystyle \text{Case~\ref{r=2}}$}};
	\draw [->,decorate, thick, decoration={snake,amplitude=.4mm,segment length=2mm,post length=1mm}] (11,5) -- (6,5);
	\node at (11,5)[rounded rectangle, draw = cyan!70!black, fill = white, scale = 0.7]{{MCP $y^u y_k y_{d+1}y_{d+2}$}};
\end{tikzpicture}
\\
\begin{tikzpicture}[scale = 0.4, trim left=2cm]
	\draw[help lines] (0,0) grid +(9,9);
	\draw[solid, thick] (0,0) -- +(9,9);
	\draw[draw = red,  thick] (0,0) -- (0,2) -- (1,2)--(1,3) -- (2,3)--(2,4)--(5,4); 
	\draw[draw = red,  thick, densely dashed] (5,4) -- (5,9); 
  	\draw[draw = cyan, opacity=0.4, line width = 3pt] (0,0) -- (0,2); 
	\draw[draw=blue!40,  thick] (0,9) -- (5,4) -- (7.5, 1.5); 
 	\node at (7.5, 1.5) [circle, fill = blue!30, scale=0.5]{$ L_{d}$}; 
	\draw[draw=green!60!black,  thick] (0,7) -- (5.5,1.5); 
 	\node at (5.5,1.5) [circle, fill = green!80!black, scale=0.5]{$ L_{k}$}; 
	\draw[draw=cyan,  thick, dashed, -latex] plot [smooth] coordinates {(2,5)  (2,7)}; 
	\draw[draw=cyan,  thick, dashed, -latex] plot [smooth] coordinates {(1,6)   (2,7)}; 
	\draw[draw=cyan,  thick, dashed, -latex] plot [smooth] coordinates {(0,7)  (2,7)}; 
	\draw[draw=cyan,  thick, dashed, -latex] plot [smooth] coordinates {(0,2) (1,3.5)  (2,4)}; 
	\draw[draw=cyan,  thick, dashed, -latex] plot [smooth] coordinates {(0,2) (0.5, 4) (1,5)}; 
	\draw[draw=cyan,  thick, dashed, -latex] plot [smooth] coordinates {(0,2) (0, 4.4)  (0,6)}; 
	\draw[draw=cyan,  thick, dashed, -latex] (2,7)--(5,9); 
	\draw[draw = cyan,  thick] (2,7) -- (5,7); 
	\draw[draw = cyan,  thick] (2,4) -- (2,5); 
	\draw[draw = cyan,  thick] (1,5) -- (1,6); 
	\draw[draw = cyan,  thick] (0,6) -- (0,7); 
	\node at (2,7) [circle, fill = cyan, scale=0.4]{}; 
	\node at (5, 7) [circle, fill = cyan, scale=0.4]{}; 
	\node at (2,5) [circle, fill = cyan, scale=0.4]{}; 
	\node at (1,6) [circle, fill = cyan, scale=0.4]{}; 
	\node at (0, 7) [circle, fill = cyan, scale=0.4]{}; 
	\node at (4.5,-1)[rounded rectangle, draw = blue!50!green, fill = white, scale = 0.7]{{$r>2\colon y^u y_{i_1}\cdots y_{i_r}$}};
\end{tikzpicture}
&\begin{tikzpicture}[scale = 0.4, trim left=-2cm ]
	\draw[help lines] (0,0) grid +(9,9);
	\draw[solid, thick] (0,0) -- +(9,9);
	\draw[draw = red,  thick] (0,0) -- (0,2) -- (1,2)--(1,3) -- (2,3)--(2,4)--(5,4); 
	\draw[draw = red,  thick, densely dashed] (5,4) -- (5,9); 
  	\draw[draw = cyan, opacity=0.4, line width = 3pt] (0,0) -- (0,2); 
	\draw[draw=blue!40,  thick] (0,9) -- (5,4) -- (7.5, 1.5); 
 	\node at (7.5, 1.5) [circle, fill = blue!30, scale=0.5]{$ L_{d}$}; 
	\draw[draw=green!60!black,  thick] (0,7) -- (5.5,1.5); 
 	\node at (5.5,1.5) [circle, fill = green!80!black, scale=0.5]{$ L_{k}$}; 
	\draw[draw=cyan,  thick, dashed, -latex] plot [smooth] coordinates {(3,4)  (3,6)}; 
	\draw[draw=cyan,  thick, dashed, -latex] plot [smooth] coordinates {(2,5)  (3,6)}; 
	\draw[draw=cyan,  thick, dashed, -latex] plot [smooth] coordinates {(1,6) (3,6)}; 
	\draw[draw=cyan,  thick, dashed, -latex] plot [smooth] coordinates {(0,2) (1,3.5)  (2,4)}; 
	\draw[draw=cyan,  thick, dashed, -latex] plot [smooth] coordinates {(0,2) (0.5, 4) (1,5)}; 
	\draw[draw=cyan,  thick, dashed, -latex] plot [smooth] coordinates {(0,2) (0, 4.4)  (0,6)}; 
	\draw[draw=cyan,  thick, dashed, -latex] (3,6)--(6,8); 
	\draw[draw=cyan,  thick, dashed] (6,6)--(6,8); 
	\draw[draw = cyan,  thick] (3,6) -- (6,6); 
	\draw[draw = cyan,  thick] (2,4) -- (3,4); 
	\draw[draw = cyan,  thick] (1,5) -- (2,5); 
	\draw[draw = cyan,  thick] (0,6) -- (1,6); 
	\node at (3,6) [circle, fill = cyan, scale=0.4]{}; 
	\node at (6,6) [circle, fill = cyan, scale=0.4]{}; 
	\node at (2,5) [circle, fill = cyan, scale=0.4]{}; 
	\node at (1,6) [circle, fill = cyan, scale=0.4]{}; 
	\node at (3,4) [circle, fill = cyan, scale=0.4]{}; 
	\node at (4.5,-1)[rounded rectangle, draw = blue!50!green, fill = white, scale = 0.7]{{$\displaystyle r>2\colon y^u y_{i_1}\cdots y_{i_r} \atop
										 \displaystyle \text{Case~\ref{rgeq3}}$}};
\end{tikzpicture}
\end{align*}
            \caption{On the left are the path of monomials $y^\beta$ for $\beta\in P_\alpha^{(k)}$. On the right are the path of the terms in Equation~\eqref{spoly_lemma_eqn2} for $r>1$}
            \label{Fig:case2} 
\end{figure}
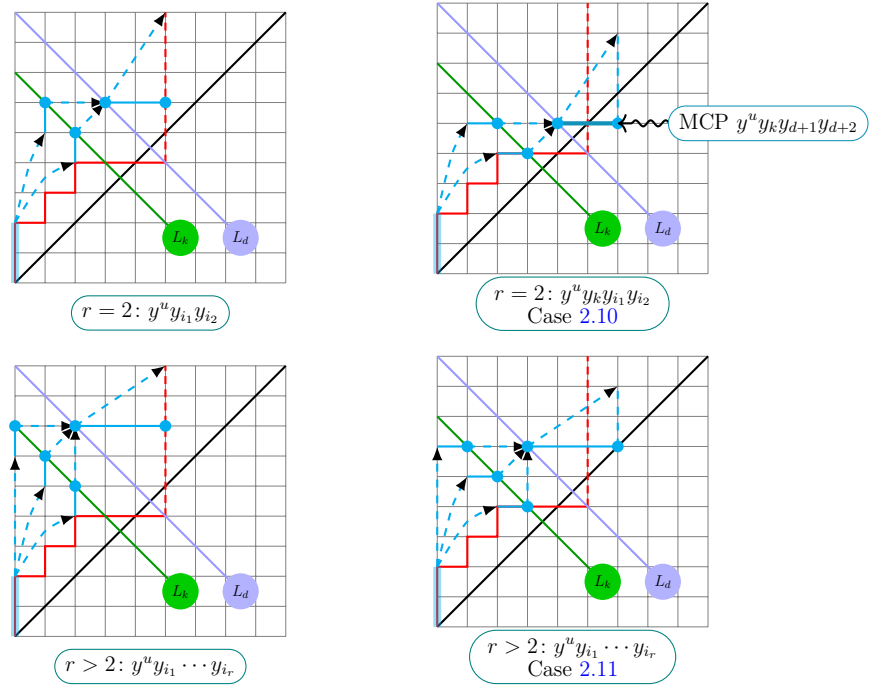 
When we multiply $y^\beta=y^{u} y_{i_1} \cdots y_{i_{r}}$ by $y_k$, the portion of the path corresponding to 
$\beta_{k+1}\beta_{k+2}\cdots\beta_n$ will be translated down by one unit and translated right by one unit.
All the paths remain above the diagonal except when $r=2$ and $y^\beta=y^{u} y_{d+1}y_{d+2}$.
All paths will match a subset of the paths we encounter in \S\ref{beta_prime_redu} as we see below. In particular, no further division
can be performed with polynomials in $F_{\ell_1(\alpha)-1}$.

\begin{case}[{\it $r=2$ in Equation~\eqref{spoly_lemma_eqn2}}\null] \label{r=2}
 As visualized in Figure~\ref{Fig:case2}, when $r=2$ in Equation~\eqref{spoly_lemma_eqn2} the paths we receive correspond exactly to the paths in Case~\ref{r=1} for which $v_k=1$. In this case, these are the new MCPs which are generated. 
\end{case} 

\begin{case}[{\it $r \geq 3$ in Equation~\eqref{spoly_lemma_eqn2}}\null] \label{rgeq3}
 Also Visualized in Figure~\ref{Fig:case2}, when $r>2$ in Equation~\eqref{spoly_lemma_eqn2} this set of paths corresponds exactly to the paths in Case~\ref{rgeq2}, for which $v_k=1$.
\end{case}

\subsection{\label{final_red}Combining all divisions}
The analysis of \S\ref{beta_prime_redu} and \S\ref{beta_redu} shows that in order to divide $S(g_\alpha, y_k^2)$ with respect to $F_{\ell_1(\alpha)-1}$,  we need to substitute Equation~\eqref{eq:thefirstr1} in Equation~\eqref{spoly_lemma_eqn2} for $r=1$. We also substitute $s=r-1$ to match up the terms. We obtain
\begin{align}
	\overline{S(g_\alpha, y_k^2)}^{F_{\ell_1(\alpha)-1}} = &
	  - \sum_{r=2}^{\ell_1(\alpha)+1} \sum_{v \in U_{\alpha,r-1}} ry^v \sum_{d+1 \leq i_1 < \cdots < i_{r} \leq n}y_{i_1} \cdots  y_{i_{r}} \label{eq:sreduction}\\ 
	  &\  + \sum_{r=2}^{\ell_1(\alpha)}\sum_{u \in U_{\alpha,r}\atop u_k=0 \hfill} y^{u}y_k \sum_{d+1 \leq i_1 < \cdots < i_r \leq n} y_{i_1} \cdots y_{i_{r}}  
	   \notag\\
	   = &
	   \sum_{r=2}^{\ell_1(\alpha)+1} \sum_{v \in U_{\alpha,r-1}\atop v_k=0\hfill} (-r)y^v \sum_{d+1 \leq i_1 < \cdots < i_{r} \leq n}y_{i_1} \cdots  y_{i_{r}} \label{eq:sreduction2}\\ 
	  &\  + \sum_{r=2}^{\ell_1(\alpha)}\sum_{v \in U_{\alpha,r-1}\atop v_k=1 \hfill} (-r+1)y^{v} \sum_{d+1 \leq i_1 < \cdots < i_r \leq n} y_{i_1} \cdots y_{i_{r}}.
	   \notag
\end{align}
In the second equality, we combined terms using the fact that for $u \in U_{\alpha,r}$ and $u_k=0$, the term $y^uy_k=y^v$ for $v \in U_{\alpha,r-1}$ and $v_k=1$. When $r=2$, the term $v \in U_{\alpha,r-1}$ and $v_k=1$ is unique and is exactly $v=u(\alpha,k)$  as in Definition~\ref{u_alpha_one_k}.

\subsection{\label{results} The Matrix of the Linear Transformation}

The process until now has focused on a single S-polynomial reduction $\overline{S(g_\alpha, y_k^2)}^{F_{\ell_1(\alpha)-1}} =\mathcal{S}_{\alpha, k} $, for a fixed $k$ and fixed $\alpha$. In this section, we examine all possible $\mathcal{S}_{\alpha, k_{j}}$, letting $k_j$ vary while keeping $\alpha$ fixed. Also, since $\alpha$ has $\ell_1(\alpha)$ entries equal to 1, we have $U_{\alpha, 1}=\big\{u(\alpha,k_i)\big| 1\le i\le \ell_1(\alpha)\big\}$ with  $|U_{\alpha,1}| = \ell_1(\alpha)$. Let
   $$A_\alpha = \ \  \Big[ \text{Coefficient of } y^{u(\alpha, k_i)}y_{d+1}y_{d+2} \text{ in } \mathcal{S}_{\alpha,k_j} \Big]_{1 \leq i,j \leq l+1}$$
When $r=2$ in Equation~\eqref{eq:sreduction2} we readily see that we have $-2$ in the diagonal of $A_\alpha$ and $-1$ off the  diagonal. 
This matrix is invertible, given explicitly by
$$
    A_\alpha  =  
    \begin{bmatrix} -2 & -1 & \cdots  & -1 \\-1 & -2 &  \cdots & -1 \\  \vdots & \vdots & \ddots  & \vdots \\ -1 & -1 & \cdots & -2 \end{bmatrix},
    \qquad
    M_\alpha =A_\alpha^{-1}  = \frac{1}{\ell_1(\alpha)+1} 
    \begin{bmatrix} -\ell_1(\alpha) & 1 & \cdots  & 1 \\1 & -\ell_1(\alpha)& \cdots & 1 \\ \vdots & \vdots  & \ddots & \vdots \\ 1 & 1 & \cdots  & -\ell_1(\alpha)  \end{bmatrix}
$$
Since $A_\alpha$ is invertible, we have that $\{\mathcal{S}_{\alpha, k_i}\}_{1 \leq i \leq \ell_1(\alpha)}$ is linearly independent. Let $V$ be the $\field$-span of $\{\mathcal{S}_{\alpha, k_i}\}_{1 \leq i \leq \ell_1(\alpha)}$ and we define $\blue{\phi}\colon V \to  V$ as 
\begin{equation} \label{lemma_phi_map2}
    \blue{\phi} (\mathcal{S}_{\alpha, k_i}) = \sum_{j=1}^{\ell_1(\alpha)} c_{ij} \mathcal{S}_{\alpha, {k_{j}}},
\end{equation}
where $M_\alpha=[c_{ij}]$. This is the linear transformation we require in  Lemma~\ref{lemma_phi_map}. To conclude the proof, we need to show that  $\blue{\phi}(\mathcal{S}_{\alpha, k_i})  = g_{\alpha^{(k_i)}}$. To compute this, we substitute Equation~\eqref{eq:sreduction2} in~\eqref{lemma_phi_map2} for each $k_j$. When $r=2$ and fixed $d+2 \leq i_1 < i_2 \leq n$, then the coefficient of 
$y^{u(\alpha, k_i)} y_{i_1} y_{i_2}$ in the expansion on $\mathcal{S}_{\alpha, k_i}$, is exactly the matrix $A_\alpha$. Therefore, when we make the linear combination using $M_\alpha$ (the  inverse of $A_\alpha$), we get
\begin{equation}\label{eq:coefr=2}
  \text{Coeff  of } y^{u(\alpha, k_{i'})} y_{i_1} y_{i_2} \text{ in  } \blue{\phi} (\mathcal{S}_{\alpha, k_{i}})  = 
  \begin{cases}
  	1 &\text{if } i=i'\,,\\
	0& \text{otherwise.}
  \end{cases}
  \end{equation}
For the other terms, fix $r, v$ and $d+1 \leq i_1 < \cdots < i_r \leq n$, such that $3 \leq r \leq \ell_1(\alpha)+1$ and  $v \in U_{\alpha,r-1}$.
We want to determine the coefficient of $y^{v} y_{i_1} \cdots y_{i_{r}}$ in $\blue{\phi} (\mathcal{S}_{\alpha, k_{i}}) $. This will depend on the value of $v_{k_i}$.

\begin{case}[{\it $v_{k_i}=0$}\null] \label{vk=0} 
Since $v_{k_i}=0$ and $v \in U_{\alpha,r-1}$, there are $(r-1)-1=r-2$ other entries of $v$ such that $\alpha_k=1$ and $v_k=0$.
On the other hand, the number of positions where $\alpha_k=v_k=1$ is $\ell_1(\alpha)-r+1$.
The coefficient of $y^{v} y_{i_1} \cdots y_{i_{r}}$ in $\blue{\phi} (\mathcal{S}_{\alpha, k_{i}}) $ is
\begin{align*}
 -rc_{ii}  + \sum_{v_{k_j}=0 \atop j\ne i}  -rc_{ij}   + \sum_{v_{k_j}=1 \atop j\ne i}  (-r+1)c_{ij}  
    &= \frac{r\ell_1(\alpha) -r(r-2) + (-r+1)(\ell_1(\alpha)-r+1)}{\ell_1(\alpha)+1} \\
    &= 1\,.
 \end{align*}
 The coefficient $-r$ and $-r+1$ are from Equation~\eqref{eq:sreduction2}, and they depend on the value of $v_k$ and the entry $[c_{ij}]=M_\alpha$.
\end{case} 

\begin{case}[{\it $v_{k_i}=1$}\null] \label{vk=1} 
Now, $\alpha_{k_i}=v_{k_i}=1$. There are $r-1$ entries of $v$ such that $\alpha_k=1$ and $v_k=0$
and $\ell_1(\alpha)-r$ entries where $\alpha_k=v_k=1$, other than the entries where $k=k_i$.
The coefficient of $y^{v} y_{i_1} \cdots y_{i_{r}}$ in $\blue{\phi} (\mathcal{S}_{\alpha, k_{i}}) $ is now
\begin{align*}
 (-r+1)c_{ii}  &+ \sum_{v_{k_j}=0 \atop j\ne i}  -rc_{ij}   + \sum_{v_{k_j}=1 \atop j\ne i}  (-r+1)c_{ij}  \\
    &= \frac{(r-1)\ell_1(\alpha) -r(r-1) + (-r+1)(\ell_1(\alpha)-r)}{\ell_1(\alpha)+1} 
    = 0\,.
 \end{align*}
\end{case}
 
It follows that
\begin{equation}  \label{phi_out1}
    \blue{\phi}(\mathcal{S}_{\alpha, k_i}) =   y^{u(\alpha, k_i)} \sum_{d+1 \leq i_1 < i_2 \leq n}y_{i_1} y_{i_2} +  \sum_{r=3}^{\ell_1(\alpha)+1} \sum_{v \in U_{\alpha,r-1}\atop v_{k_i}=0\hfill} y^v \sum_{d+1 \leq i_1 < \cdots < i_{r} \leq n}y_{i_1} \cdots  y_{i_{r}}.  
\end{equation}
This sum is non-zero since $\ell_1(\alpha)-1\le\ell \le \Floor{\frac{n-1}{2}}-1$, therefore $d+2\le n$.
We want to compare this with $g_{\alpha^{(k_i)}}$, where $\alpha^{(k_i)}=u(\alpha, k_i)\texttt{110}^{n-d-2}$ for $d=2\ell_1(\alpha)-1$.
Using Equation~\eqref{y_beta_sum},
\begin{equation} \label{gin}
	g_{\alpha^{(k_i)}} = y^{{\alpha^{(k_i)}}}+  \sum_{s=1}^{\ell_1(\alpha^{(k_i)})} \sum_{u \in U_{\alpha^{(k_i)},s}} y^u
										\sum_{d+3 \leq i_1 < \cdots < i_s \leq n} y_{i_1} \cdots y_{i_{s}}. 
\end{equation}
Note that $\ell_1(\alpha^{(k_i)})=\ell_1(\alpha)+1$ and for any $u \in U_{\alpha^{(k_i)},s}$ we must have $u_{k_i}=0$, since $\alpha^{(k_i)}_{k_i}=0$. For any other position $1\le k\le d$, we have $\alpha_k=\alpha_k^{(k_i)}$. Hence, $\alpha_k=0 \implies u_k=0 \implies v_k=0$, for any  $v\in U_{\alpha,r-1}$. For $k=d+1$ or $k=d+2$, there are no constraints on $u$ as $\alpha_{d+1}^{(k_i)}=\alpha_{d+2}^{(k_i)}=1$. Moreover, the term $ y^{{\alpha^{(k_i)}}}=y^{u(\alpha, k_i)}y_{d+1}y_{d+2}$ appears in the first sum of Equation~\eqref{phi_out1}.
This shows that every term in Equation~\eqref{gin} appears in Equation~\eqref{phi_out1}.

For the converse, note that any term $y^{u(\alpha, k_i)} y_{i_1} y_{i_2}$  for $d+1 \leq i_1 < i_2 \leq  n$, corresponds to 
$y^{u(\alpha, k_i)} y_{d+1} y_{d+2}=y^{\alpha^{(k_1)}}$;
or $y^{u(\alpha, k_i)}y_{d+1}y_{i}$ and $y^{u(\alpha, k_i)}y_{d+2}y_{i}$ for $u(\alpha, k_i)\texttt{10}, u(\alpha, k_i)\texttt{01}\in U_{\alpha^{(k_i)},1}$ and
$d+3\le i\le n$;
or $y^{u(\alpha, k_i)}y_{i_1}y_{i_2}$ for $u(\alpha, k_i)\texttt{00}\in U_{\alpha^{(k_i)},2}$ and $d+3\le i_1<i_2\le n$.
Similarly, for $v\in U_{\alpha,r-1}$ with $v_{k_i}=0$, the terms $y^v y_{i_1} \cdots  y_{i_{r}} $ for $d+1 \leq i_1 < \cdots < i_{r} \leq n$ correspond to  
$y^vy_{d+1} y_{d+2} y_{j_{3}}\cdots  y_{j_{r-1}} $, for $v\texttt{11}\in U_{\alpha^{(k_i)},r}$ and $d+3 \leq j_3 < \cdots <j_{r} \leq n$; 
or 
$y^vy_{d+1} y_{j_{2}}\cdots  y_{j_{r-1}} $ and $y^vy_{d+2} y_{j_{2}}\cdots  y_{j_{r-1}} $ for $v\texttt{01},v\texttt{10} \in U_{\alpha^{(k_i)},r+1}$ and $d+3 \leq j_2< \cdots <j_{r} \leq n$;
or
$y^v y_{j_{1}}\cdots  y_{j_{r-1}} $ for $v\texttt{00} \in U_{\alpha^{(k_i)},r+2}$ and $d+3 \leq j_1< \cdots <j_{r} \leq n$.
This shows the reverse inclusion and concludes our proof of 
 Lemma~\ref{lemma_phi_map}. \hfill$\square$

\subsection{\label{property_c}Termination of Buchberger's Algorithm}
In the previous subsections, we computed the set of remainders  of S-polynomials $\overline{S(g_\alpha, k_1)}^{F_{\ell_1(\alpha)-1}}, \dots, \overline{S(g_\alpha, k_{\ell_1(\alpha)})}^{F_{\ell_1(\alpha)-1}}$ and used Lemma~\ref{lemma_phi_map} to conclude that, $\{g_{\alpha^{(k_i)}}\}_{1\le i\le \ell_1(\alpha)}$ is obtained as an intermediate step in the Buchberger's Algorithm. If $\ell_1(\alpha)-1<\ell$, then these polynomials are already in $F_\ell$ so the division with respect to $F_\ell$ will be zero (no new polynomials). For $\ell_1(\alpha)-1=\ell$, we will obtain a new polynomial and construct $F_{\ell+1}$ in this way, as long as $\ell_1(\alpha)-1 \le \Floor{\frac{n-1}{2}}-1$.
We need to show that for $\ell = \lfloor \frac{n-1}{2} \rfloor$, the set $F_\ell$ is the reduced Gr\"obner basis.
The next lemma shows that the set
\begin{equation} \label{Eq:basis1}
	B=\big\{ y^\gamma\big| \text{\it square free and $\gamma$ stays above the diagonal}\big\}\,,
\end{equation}
$\field$-spans the quotient $R/I$ using division by $F_\ell$ only. In the next section, we will show that $\dim(R/I)$ is at least the cardinality of $B$.
From this we will conclude that $B$ is a basis, and therefore $F_\ell$ is the full reduced Gr\"obner basis of $I$. This will complete the proof of Theorem~\ref{main_theorem2}.

\begin{lemma}
	    \label{lem:Bspan}
	    Let $\ell= \lfloor \frac{n-1}{2} \rfloor$. For any monomial $y^\delta$, the remainder $\overline{y^\delta}^{F_\ell}$ is a $\field$-linear combination of $B$. In particular, $B$ spans $R/I$.
    \end{lemma} 

\begin{proof}
If $y^\delta$ is not square free, then it is divisible by $y_k^2$ for some $k$ and we get $\overline{y^\delta}^{F_\ell}=0$. 
If $y^\delta$ crosses the diagonal, 
let $k$ be the smallest integer such that, $\alpha=\delta_1\delta_2\cdots\delta_k \texttt{0}^{n-k}$ 
crosses under the diagonal. It is clear that $\alpha$ is a MCP and $\ell_1(\alpha)-1\le \ell$, hence $g_\alpha\in F_\ell$. Since we have $y^\delta=y^\alpha y^\epsilon$, we obtain
\begin{equation}\label{eq:contra}
	\overline{y^\delta}^{F_\ell}=\overline{y^\alpha y^\epsilon}^{F_\ell}=\overline{-\sum_{\beta\in P_\alpha}y^\beta y^\epsilon}^{F_\ell}.
\end{equation}
All terms $y^\beta y^\epsilon<_{lex} y^\alpha y^\epsilon$. We repeat the process on any monomials in the sum that cross the diagonal. 
We know this process terminates (division algorithm) and we are left only with a linear combination of monomials in $B$.
\end{proof}

\section{Orthogonal complement, dimension, and irreducible decomposition}\label{sec:irreducible}
In this section, we quickly introduce an inverse system to compute the orthogonal complement $H=I^\perp$ of $I$.
We have that $H\cong R/I$.
We then construct a set $B'$ of independent elements inside  $H$ that have the same cardinality as $B$ in Equation~\eqref{Eq:basis1}.
This will show that both $B$ and $B'$ are a basis of the  quotient, concluding the proof of Theorem~\ref{main_theorem2}.
 We then remark that  since $I$ is homogeneous and remains invariant under the permutation of variables, we can define an action of the symmetric group on the (graded) quotient $R/I$. Using the  irreducible representation theory of the symmetric group (see \cite{Sagan} for example), we show that $B'$  exhibits a unique irreducible representation for each homogeneous component of $R/I$. This will complete our full understanding of this quotient.

\subsection{Inverse System and Orthogonal complement of $I$}\label{sec:inverse}
At the root of commutative algebra, one finds the theory of inverse systems developed by Macaulay~\cite{Macaulay}. This is classical theory, and we
only review (without proof) the needed ingredients.
Naively, we aim to study the quotient $R/I$ via the orthogonal complement of $I$ under the following scalar product.
For $P,Q\in R$, we define
 \begin{equation}\label{eq:scalar}
     \la P, Q \ra = \Big( P\big({\textstyle \frac{ \partial}{ \partial y_1},\frac{ \partial}{ \partial y_2},\ldots,\frac{ \partial}{ \partial y_n}}\big) Q\Big)(0,0,\ldots,0).
   \end{equation}
That is, we substitute the partial derivatives in $P$ and apply this operator to $Q$, after we evaluate the resulting polynomial at $(0,0,\ldots,0)$.
This is a scalar product on $R$. We then compute the orthogonal complement $H=I^\perp$ in $R$ with respect to the scalar product in Equation~\eqref{eq:scalar}.
A wonderful lemma (a consequence of Taylor expansion) gives us the following.
\begin{lemma}\label{lem:ortho}
  $$H = \big\{ Q \in R\big| P\big({\textstyle \frac{ \partial}{ \partial y_1},\ldots,\frac{ \partial}{ \partial y_n}}\big) Q =0, \text{ $P$ generators of I}\big\}$$
  \end{lemma}
 The difference is that now $H$ is the solution set of a system of differential equations, we do not need the evaluation at zero. Since $H$ is the orthogonal complement
 of $I$, we automatically get that
 $$ H \cong R/I\,.$$
 
\subsection{Standard Young Tableaux }\label{ss:standardYoung}
To define elements in $H$, we need to introduce the notion of Standard Young Tableaux. This will also be useful  for the
representation theory of the symmetric group aspect of our investigation. We shall also see that the set $B$ in Equation~\eqref{Eq:basis1} is related to some standard Young tableaux.  
\begin{definition} Given a sequence of integers $\lambda=(\lambda_1,\lambda_2,\ldots,\lambda_r)$  such that $\lambda_1\ge \lambda_2\ge\cdots\ge\lambda_r >0$ and $n=\lambda_1+\cdots+\lambda_r$, we construct the set 
  $$D_\lambda=\big\{ (i,j)\in \Z\times\Z \big|  1\le j \le r;\ 1\le i \le \lambda_j \big\}\,.$$
We say that $\lambda$ is an {\bf partition} of $n$ and $D_\lambda$ is the {\bf diagram} of $\lambda$. A bijection
  $$T\colon \{1,2,\ldots,n\}\to D_\lambda\,,$$
  is called a {\bf Standard Young Tableaux} (abbreviated SYT) if 
  \begin{align*}
  T(i,j) < T(i+1,j) \qquad \text{ and }\qquad T(i,j) < T(i+1,j),
\end{align*}
whenever $T(i,j)$, $T(i+1,j)$ and/or $T(i,j) < T(i+1,j)$ are defined. We usually put the value $T(i,j)$ in the position $(i,j)$ forming
a {\sl tableau} in the plane. The entries  increase in rows and columns. We say that $\lambda$ is the {\bf shape} of $T$ and $n$ is the {\bf size} of $T$. We write $T\in SYT_\lambda$ in this case.
\end{definition}
\begin{example}\label{ex:tab}
	Let $\lambda=(5,3)$, the picture below is a visualization of a single $T\in SYT_\lambda$.
$$T=
\begin{tikzpicture}[scale = 0.4,baseline=5pt ]
	\draw[solid, thick] (0,0) -- (5,0);
	\draw[solid, thick] (0,1) -- (5,1);
	\draw[solid, thick] (0,2) -- (3,2);
	\draw[solid, thick] (0,0) -- (0,2);
	\draw[solid, thick] (1,0) -- (1,2);
	\draw[solid, thick] (2,0) -- (2,2);
	\draw[solid, thick] (3,0) -- (3,2);
	\draw[solid, thick] (4,0) -- (4,1);
	\draw[solid, thick] (5,0) -- (5,1);
	\node at (.5,.5) [scale=0.7] {$1$}; 
	\node at (1.5,.5) [scale=0.7] {$2$}; 
	\node at (.5,1.5) [scale=0.7] {$3$}; 
	\node at (2.5,.5) [scale=0.7] {$4$}; 
	\node at (3.5,.5) [scale=0.7] {$5$}; 
	\node at (1.5,1.5) [scale=0.7] {$6$}; 
	\node at (4.5,.5) [scale=0.7] {$7$}; 
	\node at (2.5,1.5) [scale=0.7] {$8$}; 
\end{tikzpicture}
$$
\end{example}
\begin{proposition}\label{prop:bijection}
For $k\le \lfloor \frac{n}{2} \rfloor$, let $\lambda=(n-k,k)$.  The number $f^\lambda=\big|SYT_\lambda\big|$ is equal to the number of square free monomials $y^\gamma$ that stay above the diagonal and $\ell_1(\gamma)=k$.
\end{proposition}

\begin{proof}
This is a very classical result and some version of it can be found in~\cite{Stanley}. The bijection $T\mapsto y^\gamma$ for $T\in SYT_\lambda$ is  given by $\gamma_i=0$, if $i$ lies in the first row of $T$, otherwise $\gamma_i=1$. One then checks that staying above the diagonal is equivalent to the inequalities defining standard  Young tableaux.
\end{proof}

In the proof above, let us denote by  $y^T$ the monomial we obtain from $T$. We thus have that $B$ in  Equation~\eqref{Eq:basis1} is also given by the following disjoint union
\begin{equation}\label{Eq:BinT}
	B=\biguplus_{k=0}^{ \lfloor \frac{n}{2} \rfloor} \big\{ y^T \big| T \in SYT_{(n-k,k)}\big\}\,.
\end{equation}

\subsection{Independent elements in $H$}\label{ss:specht}
We introduce some special Specht polynomials~\cite{Specht}. This will give us linearly independent polynomials in $H$.
Fix $k\le \lfloor \frac{n}{2} \rfloor$ and let $\lambda=(n-k,k)$. Given $T\in SYT_\lambda$ we let
\begin{equation}\label{eq:Gpoly}
 G_T(y_1,\ldots y_n) = (y_{T(1,1)} - y_{T(2,1)})  (y_{T(1,2)} - y_{T(2,2)}) \cdots  (y_{T(1,k)} - y_{T(2,k)}) \,.
 \end{equation}
  Here, we only define $G_T$ for shape $\lambda$ with two parts, but it can be done for any shape using Vandermonde polynomials for each column of $T$. 
  These are  the polynomials originally  defined by Specht~\cite{Specht} to construct irreducible representations of $S_n$.
\begin{example}
  Let $T$ be as in Example~\ref{ex:tab}, we have $G_T=(y_1-y_3)(y_2-y_6)(y_4-y_8)$.
\end{example}
Let
\begin{equation}\label{Eq:B2inT}
	B'=\biguplus_{k=0}^{ \lfloor \frac{n}{2} \rfloor} \big\{ G_T \big| T \in SYT_{(n-k,k)}\big\}\,.
\end{equation}

\begin{theorem}\label{thm:indB}
 We have that $B'\subseteq H=I^\perp$ is linearly independent.
\end{theorem}
\begin{proof}
 For the inclusion, using Lemma~\ref{lem:ortho}, we need to show that for $G_T\in B'$ we have
   $$P\big({\textstyle \frac{ \partial}{ \partial y_1},\ldots,\frac{ \partial}{ \partial y_n}}\big) G_T =0, $$
   for all generators $P$ of $I$. If $P=y_k^2$, then clearly this results in zero, since $G_T$ is linear in any single variable (square free).
   For $P=y_1+y_2+\cdots+y_n$, we use the product rule and reduce to the case
   $$\big({\textstyle \frac{ \partial}{ \partial y_1}+ \frac{ \partial}{ \partial y_2}+\cdots+\frac{ \partial}{ \partial y_n}}\big) (y_i-y_j)=0,$$
   for all $i,j$.
  
  To see that the  polynomials are independent, expand each polynomial $G_T$ and look for its trailing term in lex-order, we obtain $LT(G_T)=y_{T(2,1)}   y_{T(2,2)} \cdots  y_{T(2,k)}  = y^T$,
this is clear by always taking the largest variable in each term $(y_i-y_j)$. Since all trailing terms are distinct, we have  linear independence by  triangularity.
\end{proof}

\subsection{Proof of theorem}\label{ss:proofthm}
We have established that
 $$|B|=|B'|\le \dim(H) =\dim(R/I)\le |B|.$$
 The first equality is, by construction, obtained from Theorem~\ref{thm:indB}. The next equality is the isomorphism of $H\cong R/I$. The last inequality is 
 Lemma~\ref{lem:Bspan}.  We must have the equality that both $B$ and $B'$ are the basis of $R/I$, everywhere. To see that Theorem~\ref{main_theorem2}
 is true, assume we can compute new S-polynomials. 
 This would produce a new leading term and reduce the size of $B$ as a spanning set for $R/I$, which is absorbed.
 Hence,  Theorem~\ref{main_theorem2} holds true.

\subsection{Hilbert  Series of $R/I$}\label{ss:hilbert}
Since the ideal $I$ is homogeneous, there is a well-defined notion of degree on $N=R/I$. Hence $R/I=\bigoplus_{k\ge 0} N^{(k)}$, 
where $N^{(k)}$ is the homogeneous component of degree $k$ in $N$. We know that the maximal degree in $N$ is $k= \lfloor \frac{n}{2} \rfloor$.
Considering the disjoint decomposition of the basis $B$ in  Equation~\eqref{Eq:BinT}, we immediately obtain
\begin{theorem}\label{thm:hlbert}
 The Hilbert series of $N=R/I$ is given by
 $$ H_N(q) = \sum_{k\ge 0} \dim\big( N^{(k)}\big) q^k = \sum_{k=0}^{ \lfloor \frac{n}{2} \rfloor} f^{(n-k,k)} q^k\,.
 $$
\end{theorem}

\subsection{Irreducible graded decomposition of $R/I$}\label{ss:irreducible}
The symmetric group $S_n=\big\{ \sigma \big| \sigma\colon \{1,\ldots n\}\to\{1,\ldots,n\}\text{ bijection}\big\}$ acts on polynomials in $R$ with 
$\sigma P(y_1,\ldots,y_n)=P(y_{\sigma(1)},\ldots,y_{\sigma(n)})$. The generators of the ideal $I$ are mapped to the other generators of $I$ under this action. Hence, the ideal $I$ is invariant under the action of $S_n$. We can thus have a well-defined $\field$-linear action of $S_n$ on the quotient $N=R/I$. In fact, since the  action preserves degree, we have an action of $S_n$ on each graded piece $N^{(k)}$. Additionally, we are interested in describing the decomposition of that action in terms of the irreducible representation. For more details on the representation of the symmetric group see  \cite{Sagan}. For our needs, the irreducible representation of the symmetric group can be nicely constructed directly in the space of polynomials. We can construct such a representation using Specht polynomials~\cite{Specht}, which we introduced in \S\ref{ss:specht}.

It is known~\cite{Sagan,Specht} that for a fixed shape $\lambda$, the set 
\begin{equation}\label{eq:garnir}
 \big\{G_T\big| T \in SYT_\lambda\big\}
 \end{equation}
  is a basis for an irreducible representation 
of $S_n$ indexed  by   $\lambda$. 
Considering the disjoint decomposition of the basis $B'$ in  Equation~\eqref{Eq:B2inT}, we immediately obtain
\begin{theorem}
    For $k\le \lfloor \frac{n}{2} \rfloor$, we have that Equation~\eqref{eq:garnir} is a basis of $N^{(k)}$. In particular, $N^{(k)}$ is an irreducible representation of $S_n$
    indexed by $\lambda=(n-k,k)$. We can thus write
    $$ Frob_q(R/I) \sum_{k=0}^{ \lfloor \frac{n}{2} \rfloor} s_{(n-k,k)} q^k,$$
    where $s_\lambda$  is the Schur function used to encode the irreducible representation of $S_n$. Additionally, $Frob_q$ is the graded Frobenius map, which maps the representation to symmetric functions (see~\cite{Sagan,Macdonald}).
\end{theorem}

\small
\bibliographystyle{amsalpha}  
\bibliography{refrences}

\end{document}